\documentclass[12pt]{article}
\usepackage{amsmath,amssymb}

\makeatletter
\newfont{\footsc}{cmcsc10 at 8truept}
\newfont{\footbf}{cmbx10 at 8truept}
\newfont{\footrm}{cmr10 at 10truept}
\makeatother
\def\ang#1{\left\langle#1\right\rangle}

\newtheorem{theorem}{Theorem}

\newtheorem{Lem}{Lemma}

\newcommand{\qed}{\hfill \rule{0.7ex}{1.5ex}}

\newenvironment{proof}{\begin{trivlist} \item[\hskip \labelsep{\it
Proof.}]\setlength{\parindent}{0pt}}{\end{trivlist}}

\title{Barnes-type Daehee polynomials}

\author{
Dae San Kim
\thanks{
This work was supported by the National Research Foundation of Korea (NRF) grant funded by the Korean government (MOE) (No.2012R1A1A2003786).
}\\
\small Department of Mathematics, Sogang University\\[-0.8ex]
\small Seoul 121-742, Republic of Korea\\[-0.8ex]
\small \texttt{dskim@sogang.ac.kr}\\\\
Taekyun Kim
\thanks{
The present Research has been conducted by the Research Grant of Kwangwoon University in 2014.
}\\
\small Department of Mathematics, Kwangwoon University\\[-0.8ex]
\small Seoul 139-701, Republic of Korea\\[-0.8ex]
\small \texttt{tkkim@kw.ac.kr}\\\\
Takao Komatsu
\thanks{
The third author was supported in part by the Grant-in-Aid for Scientific research (C) (No.22540005), the Japan Society for the Promotion of Science.
}\\
\small Graduate School of Science and Technology, Hirosaki University\\[-0.8ex]
\small Hirosaki 036-8561, Japan\\[-0.8ex]
\small \texttt{komatsu@cc.hirosaki-u.ac.jp}\\\\
Jong-Jin Seo \\
\small Department of Applied Mathematics, Pukyong National University\\[-0.8ex]
\small Pusan 608-739, Republic of Korea\\[-0.8ex]
\small \texttt{seo2011@pknu.ac.kr}\\\\}

\date{
\small MR Subject Classifications: 05A15, 05A40, 11B68, 11B75, 65Q05}

\begin{document}
\maketitle

\begin{abstract}
In this paper, we consider Barnes-type Daehee polynomials of the first kind and of the second kind.  From the properties of Sheffer sequences of these polynomials arising from umbral calculus, we derive new and interesting identities.
\end{abstract}

\section{Introduction}

In this paper, we consider the polynomials $D_n(x|a_1,\dots,a_r)$ and $\widehat D_n(x|a_1,\dots,a_r)$ called the Barnes-type Daehee polynomials of the first kind and of the second kind, whose generating functions are given by
\begin{align}
\prod_{j=1}^r\left(\frac{\ln(1+t)}{(1+t)^{a_j}-1}\right)(1+t)^x&=\sum_{n=0}^\infty D_n(x|a_1,\dots,a_r)\frac{t^n}{n!}\,,
\label{barnesdaehee1}\\
\prod_{j=1}^r\left(\frac{(1+t)^{a_j}\ln(1+t)}{(1+t)^{a_j}-1}\right)(1+t)^x&=\sum_{n=0}^\infty\widehat D_n(x|a_1,\dots,a_r)\frac{t^n}{n!}\,,
\label{barnesdaehee2}
\end{align}
respectively, where $a_1,\dots,a_r\ne 0$.
When $x=0$, $D_n(a_1,\dots,a_r)=D_n(0|a_1,\dots,a_r)$ and $\widehat D_n(a_1,\dots,a_r)=\widehat D_n(0|a_1,\dots,a_r)$ are called the Barnes-type Daehee numbers of the first kind and of the second kind, respectively.

Recall that the Daehee polynomials of the first kind and of the second kind of order $r$, denoted by $D_n^{(r)}(x)$ and $\widehat D_n^{(r)}(x)$, respectively, are given by the generating functions to be
\begin{align*}
\left(\frac{\ln(1+t)}{t}\right)^r(1+t)^x&=\sum_{n=0}^\infty D_n^{(r)}(x)\frac{t^n}{n!}\,,\\
\left(\frac{(1+t)\ln(1+t)}{t}\right)^r(1+t)^x&=\sum_{n=0}^\infty\widehat D_n^{(r)}(x)\frac{t^n}{n!}\,,
\end{align*}
respectively. If $a_1=\cdots=a_r=1$, then $D_n^{(r)}(x)=D_n(x|\underbrace{1,\dots,1}_r)$ and $\widehat D_n^{(r)}(x)=\widehat D_n(x|\underbrace{1,\dots,1}_r)$.
Dahee polynomials were defined by the second author \cite{TKim1} and have been investigated in \cite{KKR,OCS, JK}.

In this paper, we consider Barnes-type Daehee polynomials of the first kind and of the second kind.  From the properties of Sheffer sequences of these polynomials arising from umbral calculus, we derive new and interesting identities.

\section{Umbral calculus}

Let $\mathbb C$ be the complex number field and let $\mathcal F$ be the set of all formal power series in the variable $t$:
\begin{equation}
\mathcal F=\left\{f(t)=\sum_{k=0}^\infty\frac{a_k}{k!}t^k\Bigg|a_k\in\mathbb C\right\}\,.
\label{uc:fps}
\end{equation}

Let $\mathbb P=\mathbb C[x]$ and let $\mathbb P^\ast$ be the vector space of all linear functionals on $\mathbb P$.
$\ang{L|p(x)}$ is the action of the linear functional $L$ on the polynomial $p(x)$, and
we recall that the vector space operations on $\mathbb P^\ast$ are defined by $\ang{L+M|p(x)}=\ang{L|p(x)}+\ang{M|p(x)}$, $\ang{cL|p(x)}=c\ang{L|p(x)}$, where $c$ is a complex constant in $\mathbb C$. For $f(t)\in\mathcal F$, let us define the linear functional on $\mathbb P$ by setting
\begin{equation}
\ang{f(t)|x^n}=a_n,\quad (n\ge 0).
\label{uc9}
\end{equation}
In particular,
\begin{equation}
\ang{t^k|x^n}=n!\delta_{n,k}\quad (n, k\ge 0),
\label{uc10}
\end{equation}
where $\delta_{n,k}$ is the Kronecker's symbol.

For $f_L(t)=\sum_{k=0}^\infty\frac{\ang{L|x^k}}{k!}t^k$, we have $\ang{f_L(t)|x^n}=\ang{L|x^n}$. That is, $L=f_L(t)$.
The map $L\mapsto f_L(t)$ is a vector space isomorphism from $\mathbb P^\ast$ onto $\mathcal F$. Henceforth, $\mathcal F$ denotes both the algebra of formal power series in $t$ and the vector space of all linear functionals on $\mathbb P$, and so an element $f(t)$ of $\mathcal F$ will be thought of as both a formal power series and a linear functional. We call $\mathcal F$ the {\it umbral algebra} and the {\it umbral calculus} is the study of umbral algebra.
The order $O\bigl(f(t)\bigr)$ of a power series $f(t)(\ne 0)$ is the smallest integer $k$ for which the coefficient of $t^k$ does not vanish. If $O\bigl(f(t)\bigr)=1$, then $f(t)$ is called a {\it delta series}; if $O\bigl(f(t)\bigr)=0$, then $f(t)$ is called an {\it invertible series}.
For $f(t), g(t)\in\mathcal F$ with $O\bigl(f(t)\bigr)=1$ and $O\bigl(g(t)\bigr)=0$, there exists a unique sequence $s_n(x)$ ($\deg s_n(x)=n$) such that $\ang{g(t)f(t)^k|s_n(x)}=n!\delta_{n,k}$, for $n,k\ge 0$.
Such a sequence $s_n(x)$ is called the {\it Sheffer sequence} for $\bigl(g(t), f(t)\bigr)$ which is denoted by $s_n(x)\sim\bigl(g(t), f(t)\bigr)$.

For $f(t), g(t)\in\mathcal F$ and $p(x)\in\mathbb P$, we have
\begin{equation}
\ang{f(t)g(t)|p(x)}=\ang{f(t)|g(t)p(x)}=\ang{g(t)|f(t)p(x)}
\label{uc11}
\end{equation}
and
\begin{equation}
f(t)=\sum_{k=0}^\infty\ang{f(t)|x^k}\frac{t^k}{k!},\quad
p(x)=\sum_{k=0}^\infty\ang{t^k|p(x)}\frac{x^k}{k!}
\label{uc12}
\end{equation}
(\cite[Theorem 2.2.5]{Roman}). Thus, by (\ref{uc12}), we get
\begin{equation}
t^k p(x)=p^{(k)}(x)=\frac{d^k p(x)}{d x^k}\quad\hbox{and}\quad e^{yt}p(x)=p(x+y).
\label{uc13}
\end{equation}

Sheffer sequences are characterized in the generating function (\cite[Theorem 2.3.4]{Roman}).

\begin{Lem}
The sequence $s_n(x)$ is Sheffer for $\big(g(t),f(t)\bigr)$ if and only if
$$
\frac{1}{g\bigl(\bar f(t)\bigr)}e^{y\bar f(t)}=\sum_{k=0}^\infty\frac{s_k(y)}{k!}t^k\quad(y\in\mathbb C)\,,
$$
where $\bar f(t)$ is the compositional inverse of $f(t)$.
\label{th234}
\end{Lem}

For $s_n(x)\sim\bigl(g(t), f(t)\bigr)$, we have the following equations (\cite[Theorem 2.3.7, Theorem 2.3.5, Theorem 2.3.9]{Roman}):
\begin{align}
f(t)s_n(x)&=n s_{n-1}(x)\quad (n\ge 0),
\label{uc15}\\
s_n(x)&=\sum_{j=0}^n\frac{1}{j!}\ang{g\bigl(\bar f(t)\bigr)^{-1}\bar f(t)^j|x^n}x^j,
\label{uc16}\\
s_n(x+y)&=\sum_{j=0}^n\binom{n}{j}s_j(x)p_{n-j}(y)\,,
\label{uc17}
\end{align}
where $p_n(x)=g(t)s_n(x)$.

Assume that $p_n(x)\sim\bigl(1, f(t)\bigr)$ and $q_n(x)\sim\bigl(1, g(t)\bigr)$. Then the transfer formula (\cite[Corollary 3.8.2]{Roman}) is given by
$$
q_n(x)=x\left(\frac{f(t)}{g(t)}\right)^n x^{-1}p_n(x)\quad (n\ge 1).
$$
For $s_n(x)\sim\bigl(g(t), f(t)\bigr)$ and $r_n(x)\sim\bigl(h(t), l(t)\bigr)$, assume that
$$
s_n(x)=\sum_{m=0}^n C_{n,m}r_m(x)\quad (n\ge 0)\,.
$$
Then we have (\cite[p.132]{Roman})
\begin{equation}
C_{n,m}=\frac{1}{m!}\ang{\frac{h\bigl(\bar f(t)\bigr)}{g\bigl(\bar f(t)\bigr)}l\bigl(\bar f(t)\bigr)^m\Bigg|x^n}\,.
\label{uc19}
\end{equation}

\section{Main results}

We now note that $D_n(x|a_1,\dots,a_r)$ is the Sheffer sequence for
$$
g(t)=\prod_{j=1}^r\left(\frac{e^{a_j t}-1}{t}\right)\quad\hbox{and}\quad
f(t)=e^t-1.
$$
So,
\begin{equation}
D_n(x|a_1,\dots,a_r)\sim\left(\prod_{j=1}^r\left(\frac{e^{a_j t}-1}{t}\right),
e^t-1\right)\,.
\label{10b}
\end{equation}
$\widehat D_n(x|a_1,\dots,a_r)$ is the Sheffer sequences for
$$
g(t)=\prod_{j=1}^r\left(\frac{e^{a_j t}-1}{t e^{a_j t}}\right)\quad\hbox{and}\quad
f(t)=e^t-1.
$$
So,
\begin{equation}
\widehat D_n(x|a_1,\dots,a_r)\sim\left(\prod_{j=1}^r\left(\frac{e^{a_j t}-1}{t e^{a_j t}}\right), e^t-1\right)\,.
\label{11b}
\end{equation}


\subsection{Explicit expressions}

Recall that Barnes' multiple Bernoulli polynomials $B_n(x|a_1,\dots,a_r)$ are defined by the generating function as
\begin{equation}
\frac{t^r}{\prod_{j=1}^r(e^{a_j t}-1)}e^{x t}=\sum_{n=0}^\infty B_n(x|a_1,\dots,a_r)\frac{t^n}{n!}\,,
\label{bmb}
\end{equation}
where $a_1,\dots,a_r\ne 0$ (\cite{JKKLPR, Kim3, BKKL}).
Let $(n)_j=n(n-1)\cdots(n-j+1)$ ($j\ge 1$) with $(n)_0=1$. The (signed) Stirling numbers of the first kind $S_1(n,m)$ are defined by
$$
(x)_n=\sum_{m=0}^n S_1(n,m)x^m\,.
$$

\begin{theorem}
\begin{align}
D_n(x|a_1,\dots,a_r)&=\sum_{m=0}^n S_1(n,m)B_m(x|a_1,\dots,a_r)
\label{100a}\\
&=\sum_{j=0}^n\left(\sum_{l=j}^n\binom{n}{l}S_1(l,j)D_{n-l}(a_1,\dots,a_r)\right)x^j
\label{100b}\\
&=\sum_{m=0}^n\binom{n}{m}D_{n-m}(a_1,\dots,a_r)(x)_m\,,
\label{100c}
\end{align}
\begin{align}
\widehat D_n(x|a_1,\dots,a_r)&=\sum_{m=0}^n S_1(n,m)B_m(x+a_1+\cdots+a_r|a_1,\dots,a_r)
\label{100adash}\\
&=\sum_{j=0}^n\left(\sum_{l=j}^n\binom{n}{l}S_1(l,j)\widehat D_{n-l}(a_1,\dots,a_r)\right)x^j
\label{100bdash}\\
&=\sum_{m=0}^n\binom{n}{m}\widehat D_{n-m}(a_1,\dots,a_r)(x)_m\,.
\label{100cdash}
\end{align}
\label{th100}
\end{theorem}

\begin{proof}
Since
\begin{equation}
\prod_{j=1}^r\left(\frac{e^{a_j t}-1}{t}\right)D_n(x|a_1,\dots,a_r)\sim(1, e^t-1)
\label{10ab}
\end{equation}
and
\begin{equation}
(x)_n\sim(1,e^t-1)\,,
\label{51}
\end{equation}
we have
\begin{align*}
D_n(x|a_1,\dots,a_r)&=\prod_{j=1}^r\left(\frac{t}{e^{a_j t}-1}\right)(x)_n\\
&=\sum_{m=0}^n S_1(n,m)\prod_{j=1}^r\left(\frac{t}{e^{a_j t}-1}\right)x^m\\
&=\sum_{m=0}^n S_1(n,m)B_m(x|a_1,\dots,a_r)\,.
\end{align*}
So, we get (\ref{100a}).

Similarly, by
\begin{equation}
\prod_{j=1}^r\left(\frac{e^{a_j t}-1}{t e^{a_j t}}\right)\widehat D_n(x|a_1,\dots,a_r)\sim(1, e^t-1)
\label{11bc}
\end{equation}
and (\ref{51}), we have
\begin{align*}
\widehat D_n(x|a_1,\dots,a_r)&=\prod_{j=1}^r\left(\frac{t e^{a_j t}}{e^{a_j t}-1}\right)(x)_n\\
&=\sum_{m=0}^n S_1(n,m)\prod_{j=1}^r\left(\frac{t e^{a_j t}}{e^{a_j t}-1}\right)x^m\\
&=\sum_{m=0}^n S_1(n,m)e^{(a_1+\cdots+a_r)t}\prod_{j=1}^r\left(\frac{t}{e^{a_j t}-1}\right)x^m\\
&=\sum_{m=0}^n S_1(n,m)B_m(x+a_1+\cdots+a_r|a_1,\dots,a_r)\,.
\end{align*}
So, we get (\ref{100adash}).


By (\ref{uc16}) with (\ref{10b}), we get
$$
D_n(x|a_1,\dots,a_r)=\sum_{j=0}^n\frac{1}{j!}\ang{\prod_{j=1}^r\left(\frac{\ln(1+t)}{(1+t)^{a_j}-1}\right)\bigl(\ln(1+t)\bigr)^j\Big|x^n}x^j\,.
$$
Since
\begin{align*}
&\ang{\prod_{j=1}^r\left(\frac{\ln(1+t)}{(1+t)^{a_j}-1}\right)\bigl(\ln(1+t)\bigr)^j\Big|x^n}\\
&=\ang{\prod_{j=1}^r\left(\frac{\ln(1+t)}{(1+t)^{a_j}-1}\right)\Big|\bigl(\ln(1+t)\bigr)^j x^n}\\
&=\ang{\prod_{j=1}^r\left(\frac{\ln(1+t)}{(1+t)^{a_j}-1}\right)\Big|j!\sum_{l=j}^\infty S_1(l,j)\frac{t^l}{l!}x^n}\\
&=j!\sum_{l=j}^n\binom{n}{l}S_1(l,j)\ang{\prod_{j=1}^r\left(\frac{\ln(1+t)}{(1+t)^{a_j}-1}\right)\Big|x^{n-l}}\\
&=j!\sum_{l=j}^n\binom{n}{l}S_1(l,j)\ang{\sum_{i=0}^\infty D_i(a_1,\dots,a_r)\frac{t^i}{i!}\Big|x^{n-l}}\\
&=j!\sum_{l=j}^n\binom{n}{l}S_1(l,j)D_{n-l}(a_1,\dots,a_r)\,,
\end{align*}
we obtain (\ref{100b}).

Similarly, by (\ref{uc16}) with (\ref{11b}), we get
$$
\widehat D_n(x|a_1,\dots,a_r)=\sum_{j=0}^n\frac{1}{j!}\ang{\prod_{j=1}^r\left(\frac{(1+t)^{a_j}\ln(1+t)}{(1+t)^{a_j}-1}\right)\bigl(\ln(1+t)\bigr)^j\Big|x^n}x^j\,.
$$
Since
\begin{align*}
&\ang{\prod_{j=1}^r\left(\frac{(1+t)^{a_j}\ln(1+t)}{(1+t)^{a_j}-1}\right)\bigl(\ln(1+t)\bigr)^j\Big|x^n}\\
&=\ang{\prod_{j=1}^r\left(\frac{(1+t)^{a_j}\ln(1+t)}{(1+t)^{a_j}-1}\right)\Big|\bigl(\ln(1+t)\bigr)^j x^n}\\
&=j!\sum_{l=j}^n\binom{n}{l}S_1(l,j)\ang{\prod_{j=1}^r\left(\frac{(1+t)^{a_j}\ln(1+t)}{(1+t)^{a_j}-1}\right)\Big|x^{n-l}}\\
&=j!\sum_{l=j}^n\binom{n}{l}S_1(l,j)\ang{\sum_{i=0}^\infty\widehat D_i(a_1,\dots,a_r)\frac{t^i}{i!}\Big|x^{n-l}}\\
&=j!\sum_{l=j}^n\binom{n}{l}S_1(l,j)\widehat D_{n-l}(a_1,\dots,a_r)\,,
\end{align*}
we obtain (\ref{100bdash}).


Next, we obtain that
\begin{align*}
D_n(y|a_1,\dots,a_r)&=\ang{\sum_{i=0}^\infty D_i(y|a_1,\dots,a_r)\frac{t^i}{i!}\Big|x^n}\\
&=\ang{\prod_{j=1}^r\left(\frac{\ln(1+t)}{(1+t)^{a_j}-1}\right)(1+t)^y\Big|x^n}\\
&=\ang{\prod_{j=1}^r\left(\frac{\ln(1+t)}{(1+t)^{a_j}-1}\right)\Big|\sum_{m=0}^\infty(y)_m\frac{t^m}{m!}x^n}\\
&=\sum_{m=0}^n(y)_m\binom{n}{m}\ang{\prod_{j=1}^r\left(\frac{\ln(1+t)}{(1+t)^{a_j}-1}\right)\Big|x^{n-m}}\\
&=\sum_{m=0}^n\binom{n}{m}D_{n-m}(a_1,\dots,a_r)(y)_m\,.
\end{align*}
Thus, we get the identity (\ref{100c}).

Similarly,
\begin{align*}
\widehat D_n(y|a_1,\dots,a_r)&=\ang{\sum_{i=0}^\infty\widehat D_i(y|a_1,\dots,a_r)\frac{t^i}{i!}\Big|x^n}\\
&=\ang{\prod_{j=1}^r\left(\frac{(1+t)^{a_j}\ln(1+t)}{(1+t)^{a_j}-1}\right)(1+t)^y\Big|x^n}\\
&=\ang{\prod_{j=1}^r\left(\frac{(1+t)^{a_j}\ln(1+t)}{(1+t)^{a_j}-1}\right)\Big|\sum_{m=0}^\infty(y)_m\frac{t^m}{m!}x^n}\\
&=\sum_{m=0}^n(y)_m\binom{n}{m}\ang{\prod_{j=1}^r\left(\frac{(1+t)^{a_j}\ln(1+t)}{(1+t)^{a_j}-1}\right)\Big|x^{n-m}}\\
&=\sum_{m=0}^n\binom{n}{m}\widehat D_{n-m}(a_1,\dots,a_r)(y)_m\,.
\end{align*}
Thus, we get the identity (\ref{100cdash}).
\qed\end{proof}


\subsection{Sheffer identity}

\begin{theorem}
\begin{align}
D_n(x+y|a_1,\dots,a_r)&=\sum_{j=0}^n\binom{n}{j}D_j(x|a_1,\dots,a_r)(y)_{n-j}\,,
\label{20}\\
\widehat D_n(x+y|a_1,\dots,a_r)&=\sum_{j=0}^n\binom{n}{j}\widehat D_j(x|a_1,\dots,a_r)(y)_{n-j}\,.
\label{21}
\end{align}
\label{th200}
\end{theorem}

\begin{proof}
By (\ref{10b}) with
\begin{align*}
p_n(x)&=\prod_{j=1}^r\left(\frac{e^{a_j t}-1}{t}\right)D_n(x|a_1,\dots,a_r)\\
&=(x)_n\sim(1,e^t-1)\,,
\end{align*}
using (\ref{uc17}), we have (\ref{20}).

By (\ref{11b}) with
\begin{align*}
p_n(x)&=\prod_{j=1}^r\left(\frac{e^{a_j t}-1}{t e^{a_j t}}\right)\widehat D_n(x|a_1,\dots,a_r)\\
&=(x)_n\sim(1,e^t-1)\,,
\end{align*}
using (\ref{uc17}), we have (\ref{21}).
\qed\end{proof}


\subsection{Difference relations}

\begin{theorem}
\begin{align}
D_n(x+1|a_1,\dots,a_r)-D_n(x|a_1,\dots,a_r)&=n D_{n-1}(x|a_1,\dots,a_r)\,,
\label{25a}\\
\widehat D_n(x+1|a_1,\dots,a_r)-\widehat D_n(x|a_1,\dots,a_r)&=n\widehat D_{n-1}(x|a_1,\dots,a_r)\,.
\label{25b}
\end{align}
\end{theorem}
\begin{proof}
By (\ref{uc15}) with (\ref{10b}), we get
$$
(e^t-1)D_n(x|a_1,\dots,a_r)=n D_{n-1}(x|a_1,\dots,a_r)\,.
$$
By (\ref{uc13}), we have (\ref{25a}).

Similarly, by (\ref{uc15}) with (\ref{11b}), we get
$$
(e^t-1)\widehat D_n(x|a_1,\dots,a_r)=n\widehat D_{n-1}(x|a_1,\dots,a_r)\,.
$$
By (\ref{uc13}), we have (\ref{25b}).
\qed\end{proof}


\subsection{Recurrence}

\begin{theorem}
\begin{align}
D_{n+1}(x|a_1,\dots,a_r)
&=x D_n(x-1|a_1,\dots,a_r)\notag\\
&\quad -\sum_{m=0}^n\left(\sum_{i=m}^n\sum_{l=i}^n\sum_{j=1}^r\frac{1}{i+1}\binom{n}{l}\binom{i+1}{m}S_1(l,i)\right.\notag\\
&\qquad\left. \times B_{i+1-m}(-a_j)^{i+1-m}D_{n-l}(a_1,\dots,a_r)\right)(x-1)^m\,, \label{30}\\
\widehat D_{n+1}(x|a_1,\dots,a_r)
&=\left(x+\sum_{j=1}^r a_j\right)\widehat D_n(x-1|a_1,\dots,a_r)\notag\\
&\quad -\sum_{m=0}^n\left(\sum_{i=m}^n\sum_{l=i}^n\sum_{j=1}^r\frac{1}{i+1}\binom{n}{l}\binom{i+1}{m}S_1(l,i)\right.\notag\\
&\qquad\left. \times B_{i+1-m}(-a_j)^{i+1-m}\widehat D_{n-l}(a_1,\dots,a_r)\right)(x-1)^m\,, \label{31}
\end{align}
where $B_n$ is the $n$th ordinary Bernoulli number.
\label{th300}
\end{theorem}
\begin{proof}
By applying
\begin{equation}
s_{n+1}(x)=\left(x-\frac{g'(t)}{g(t)}\right)\frac{1}{f'(t)}s_n(x)
\label{35}
\end{equation}
(\cite[Corollary 3.7.2]{Roman}) with (\ref{10b}), we get
$$
D_{n+1}(x|a_1,\dots,a_r)=x D_n(x-1|a_1,\dots,a_r)-e^{-t}\frac{g'(t)}{g(t)}D_n(x|a_1,\dots,a_r)\,.
$$
Now,
\begin{align*}
\frac{g'(t)}{g(t)}&=(\ln g(t))'\\
&=\left(\sum_{j=1}^r\ln(e^{a_j t}-1)-r\ln t\right)'\\
&=\sum_{j=1}^r\frac{a_j e^{a_j t}}{e^{a_j t}-1}-\frac{r}{t}\\
&=\frac{\sum_{j=1}^r\prod_{i\ne j}(e^{a_i t}-1)(a_j t e^{a_j t}-e^{a_j t}+1)}{t\prod_{j=1}^r(e^{a_j t}-1)}\,.
\end{align*}
Since
\begin{align*}
\sum_{j=1}^r\frac{a_j t e^{a_j t}}{e^{a_j t}-1}-r
&=\frac{\sum_{j=1}^r\prod_{i\ne j}(e^{a_i t}-1)(a_j t e^{a_j t}-e^{a_j t}+1)}{\prod_{j=1}^r(e^{a_j t}-1)}\\
&=\frac{\frac{1}{2}\bigl(\sum_{j=1}^r a_1\cdots a_{j-1}a_j^2 a_{j+1}\cdots a_r\bigr)t^{r+1}+\cdots}{(a_1\cdots a_r)t^r+\cdots}\\
&=\frac{1}{2}\left(\sum_{j=1}^r a_j\right)t+\cdots
\end{align*}
is a series with order$\ge 1$, by (\ref{100b}) we have
\begin{align*}
&D_{n+1}(x|a_1,\dots,a_r)\\
&=x D_n(x-1|a_1,\dots,a_r)-e^{-t}\frac{\sum_{j=1}^r\frac{a_j t e^{a_j t}}{e^{a_j t}-1}-r}{t}D_n(x|a_1,\dots,a_r)\\
&=x D_n(x-1|a_1,\dots,a_r)-e^{-t}\frac{\sum_{j=1}^r\frac{a_j t e^{a_j t}}{e^{a_j t}-1}-r}{t}
\left(\sum_{i=0}^n\sum_{l=i}^n\binom{n}{l}S_1(l,i)D_{n-l}(a_1,\dots,a_r)x^i\right)\\
&=x D_n(x-1|a_1,\dots,a_r)\\
&\quad -\sum_{i=0}^n\sum_{l=i}^n\binom{n}{l}S_1(l,i)D_{n-l}(a_1,\dots,a_r)
e^{-t}\left(\sum_{j=1}^r\frac{a_j t e^{a_j t}}{e^{a_j t}-1}-r\right)\frac{x^{i+1}}{i+1}\,.
\end{align*}
Since
\begin{align}
e^{-t}\left(\sum_{j=1}^r\frac{a_j t e^{a_j t}}{e^{a_j t}-1}-r\right)x^{i+1}
&=e^{-t}\left(\sum_{j=1}^r\sum_{m=0}^\infty\frac{(-1)^m B_m a_j^m}{m!}t^m-r\right)x^{i+1}\notag\\
&=e^{-t}\left(\sum_{j=1}^r\sum_{m=0}^{i+1}\binom{i+1}{m}B_m(-a_j)^m x^{i+1-m}-r x^{i+1}\right)\notag\\
&=\sum_{j=1}^r\sum_{m=1}^{i+1}\binom{i+1}{m}B_m(-a_j)^m(x-1)^{i+1-m}\notag\\
&=\sum_{j=1}^r\sum_{m=0}^{i}\binom{i+1}{m}B_{i+1-m}(-a_j)^{i+1-m}(x-1)^m\,,
\label{38}
\end{align}
we have
\begin{align*}
D_{n+1}(x|a_1,\dots,a_r)
&=x D_n(x-1|a_1,\dots,a_r)\\
&\quad -\sum_{i=0}^n\sum_{l=i}^n\sum_{j=1}^r\sum_{m=0}^i\frac{1}{i+1}\binom{n}{l}\binom{i+1}{m}S_1(l,i)\\
&\qquad \times B_{i+1-m}(-a_j)^{i+1-m}D_{n-l}(a_1,\dots,a_r)(x-1)^m\\
&=x D_n(x-1|a_1,\dots,a_r)\\
&\quad -\sum_{m=0}^n\left(\sum_{i=m}^n\sum_{l=i}^n\sum_{j=1}^r\frac{1}{i+1}\binom{n}{l}\binom{i+1}{m}S_1(l,i)\right.\\
&\qquad\left. \times B_{i+1-m}(-a_j)^{i+1-m}D_{n-l}(a_1,\dots,a_r)\right)(x-1)^m\,,
\end{align*}
which is the identity (\ref{30}).

Next, by applying  (\ref{35}) with (\ref{11b}), we get
$$
\widehat D_{n+1}(x|a_1,\dots,a_r)=x\widehat D_n(x-1|a_1,\dots,a_r)-e^{-t}\frac{g'(t)}{g(t)}\widehat D_n(x|a_1,\dots,a_r)\,.
$$
Now,
\begin{align*}
\frac{g'(t)}{g(t)}&=(\ln g(t))'\\
&=\left(\sum_{j=1}^r\ln(e^{a_j t}-1)-r\ln t-\biggl(\sum_{j=1}^r a_j\biggr)t\right)'\\
&=\sum_{j=1}^r\frac{a_j e^{a_j t}}{e^{a_j t}-1}-\frac{r}{t}-\sum_{j=1}^r a_j\,.
\end{align*}
By (\ref{100bdash}) we have
\begin{align*}
&\widehat D_{n+1}(x|a_1,\dots,a_r)\\
&=\left(x+\sum_{j=1}^r a_j\right)\widehat D_n(x-1|a_1,\dots,a_r)-e^{-t}\frac{\sum_{j=1}^r\frac{a_j t e^{a_j t}}{e^{a_j t}-1}-r}{t}\widehat D_n(x|a_1,\dots,a_r)\\
&=\left(x+\sum_{j=1}^r a_j\right)\widehat D_n(x-1|a_1,\dots,a_r)\\
&\quad -e^{-t}\frac{\sum_{j=1}^r\frac{a_j t e^{a_j t}}{e^{a_j t}-1}-r}{t}
\left(\sum_{i=0}^n\sum_{l=i}^n\binom{n}{l}S_1(l,i)\widehat D_{n-l}(a_1,\dots,a_r)x^i\right)\\
&=\left(x+\sum_{j=1}^r a_j\right)\widehat D_n(x-1|a_1,\dots,a_r)\\
&\quad -\sum_{i=0}^n\sum_{l=i}^n\binom{n}{l}S_1(l,i)\widehat D_{n-l}(a_1,\dots,a_r)
e^{-t}\left(\sum_{j=1}^r\frac{a_j t e^{a_j t}}{e^{a_j t}-1}-r\right)\frac{x^{i+1}}{i+1}\,.
\end{align*}
By (\ref{38}),  we have the identity (\ref{31}).
\qed\end{proof}

\subsection{Differentiation}

\begin{theorem}
\begin{align}
\frac{d}{dx}D_n(x|a_1,\dots,a_r)
&=n!\sum_{l=0}^{n-1}\frac{(-1)^{n-l-1}}{l!(n-l)}D_l(x|a_1,\dots,a_r)\,,
\label{500a}\\
\frac{d}{dx}\widehat D_n(x|a_1,\dots,a_r)
&=n!\sum_{l=0}^{n-1}\frac{(-1)^{n-l-1}}{l!(n-l)}\widehat D_l(x|a_1,\dots,a_r)\,.
\label{500b}
\end{align}
\label{th500}
\end{theorem}

\begin{proof}
We shall use
$$
\frac{d}{dx}s_n(x)=\sum_{l=0}^{n-1}\binom{n}{l}\ang{\bar f(t)|x^{n-l}}s_l(x)
$$
({\it Cf.} \cite[Theorem 2.3.12]{Roman}).
Since
\begin{align*}
\ang{\bar f(t)|x^{n-l}}&=\ang{\ln(1+t)|x^{n-l}}\\
&=\ang{\sum_{m=1}^\infty\frac{(-1)^{m-1}t^m}{m}\Big|x^{n-l}}\\
&=\sum_{m=1}^{n-l}\frac{(-1)^{m-1}}{m}\ang{t^m|x^{n-l}}\\
&=\sum_{m=1}^{n-l}\frac{(-1)^{m-1}}{m}(n-l)!\delta_{m,n-l}\\
&=(-1)^{n-l-1}(n-l-1)!\,,
\end{align*}
with (\ref{10b}), we have
\begin{align*}
\frac{d}{dx}D_n(x|a_1,\dots,a_r)&=\sum_{l=0}^{n-1}\binom{n}{l}(-1)^{n-l-1}(n-l-1)!D_l(x|a_1,\dots,a_r)\\
&=n!\sum_{l=0}^{n-1}\frac{(-1)^{n-l-1}}{l!(n-l)}D_l(x|a_1,\dots,a_r)\,,
\end{align*}
which is the identity (\ref{500a}).
Similarly, with (\ref{11b}), we have the identity (\ref{500b}).
\qed\end{proof}


\subsection{More relations}

The classical Cauchy numbers $c_n$ are defined by
$$
\frac{t}{\ln(1+t)}=\sum_{n=0}^\infty c_n\frac{t^n}{n!}
$$
(see e.g. \cite{Com,Ko1}).

\begin{theorem}
\begin{align}
D_n(x|a_1,\dots,a_r)
&=x D_{n-1}(x-1|a_1,\dots,a_r)\notag\\
&\quad +\frac{r}{n}\sum_{l=0}^n\binom{n}{l}c_l D_{n-l}(x-1|a_1,\dots,a_r)\notag\\
&\quad -\frac{1}{n}\sum_{j=1}^r\sum_{l=0}^n \binom{n}{l}a_j c_l D_{n-l}(x+a_j-1|a_1,\dots,a_r,a_j)\,,
\label{40}\\
\widehat D_n(x|a_1,\dots,a_r)
&=\left(x+\sum_{j=1}^r a_j\right)\widehat D_{n-1}(x-1|a_1,\dots,a_r)\notag\\
&\quad +\frac{r}{n}\sum_{l=0}^n\binom{n}{l}c_l\widehat D_{n-l}(x-1|a_1,\dots,a_r)\notag\\
&\quad -\frac{1}{n}\sum_{j=1}^r\sum_{l=0}^n \binom{n}{l}a_j c_l\widehat D_{n-l}(x-1|a_1,\dots,a_r,a_j)\,.
\label{42}
\end{align}
\label{th40}
\end{theorem}
\begin{proof}
For $n\ge 1$, we have
\begin{align*}
D_n(y|a_1,\dots,a_r)&=\ang{\sum_{l=0}^\infty D_l(y|a_1,\dots,a_r)\frac{t^l}{l!}\Big|x^n}\\
&=\ang{\prod_{j=1}^r\left(\frac{\ln(1+t)}{(1+t)^{a_j}-1}\right)(1+t)^y\Big|x^n}\\
&=\ang{\partial_t\left(\prod_{j=1}^r\left(\frac{\ln(1+t)}{(1+t)^{a_j}-1}\right)(1+t)^y\right)\Big|x^{n-1}}\\
&=\ang{\prod_{j=1}^r\left(\frac{\ln(1+t)}{(1+t)^{a_j}-1}\right)\left(\partial_t(1+t)^y\right)\Big|x^{n-1}}\\
&\quad +\ang{\left(\partial_t\prod_{j=1}^r\left(\frac{\ln(1+t)}{(1+t)^{a_j}-1}\right)\right)(1+t)^y\Big|x^{n-1}}\\
&=y D_{n-1}(y-1|a_1,\dots,a_r)\\
&\quad +\ang{\left(\partial_t\prod_{j=1}^r\left(\frac{\ln(1+t)}{(1+t)^{a_j}-1}\right)\right)(1+t)^y\Big|x^{n-1}}\,.
\end{align*}
Observe that
\begin{align*}
&\partial_t\prod_{j=1}^r\left(\frac{\ln(1+t)}{(1+t)^{a_j}-1}\right)\\
&=\sum_{j=1}^r\prod_{i\ne j}\left(\frac{\ln(1+t)}{(1+t)^{a_i}-1}\right)
\frac{\frac{1}{1+t}\bigl((1+t)^{a_j}-1\bigr)-\ln(1+t)\bigl(a_j(1+t)^{a_j-1}\bigr)}{\bigl((1+t)^{a_j}-1\bigr)^2}\\
&=\frac{1}{1+t}\prod_{i=1}^r\left(\frac{\ln(1+t)}{(1+t)^{a_i}-1}\right)\sum_{j=1}^r\left(\frac{1}{\ln(1+t)}-\frac{a_j(1+t)^{a_j}}{(1+t)^{a_j}-1}\right)\\
&=\frac{1}{1+t}\prod_{i=1}^r\left(\frac{\ln(1+t)}{(1+t)^{a_i}-1}\right)\frac{\sum_{j=1}^r\left(\frac{t}{\ln(1+t)}-\frac{a_j t(1+t)^{a_j}}{(1+t)^{a_j}-1}\right)}{t}\,.
\end{align*}
Since
$$
\sum_{j=1}^r\left(\frac{t}{\ln(1+t)}-\frac{a_j t(1+t)^{a_j}}{(1+t)^{a_j}-1}\right)
=-\frac{1}{2}\left(\sum_{j=1}^r a_j\right)t+\cdots
$$
is a series with order$(\ge 1)$, we have
\begin{align*}
&\ang{\left(\partial_t\prod_{i=1}^r\left(\frac{\ln(1+t)}{(1+t)^{a_i}-1}\right)\right)(1+t)^y\Big|x^{n-1}}\\
&=\ang{\prod_{i=1}^r\left(\frac{\ln(1+t)}{(1+t)^{a_i}-1}\right)(1+t)^{y-1}\Big|\frac{\sum_{j=1}^r\left(\frac{t}{\ln(1+t)}-\frac{a_j t(1+t)^{a_j}}{(1+t)^{a_j}-1}\right)}{t}x^{n-1}}\\
&=\frac{1}{n}\sum_{j=1}^r\ang{\prod_{i=1}^r\left(\frac{\ln(1+t)}{(1+t)^{a_i}-1}\right)(1+t)^{y-1}\Big|\left(\frac{t}{\ln(1+t)}-\frac{a_j t(1+t)^{a_j}}{(1+t)^{a_j}-1}\right)x^n}\\
&=\frac{r}{n}\ang{\prod_{i=1}^r\left(\frac{\ln(1+t)}{(1+t)^{a_i}-1}\right)(1+t)^{y-1}\Big|\frac{t}{\ln(1+t)}x^n}\\
&\quad -\frac{1}{n}\sum_{j=1}^r a_j\ang{\frac{\ln(1+t)}{(1+t)^{a_j}-1}\prod_{i=1}^r\left(\frac{\ln(1+t)}{(1+t)^{a_i}-1}\right)(1+t)^{y+a_j-1}\Big|\frac{t}{\ln(1+t)}x^n}\\
&=\frac{r}{n}\ang{\prod_{i=1}^r\left(\frac{\ln(1+t)}{(1+t)^{a_i}-1}\right)(1+t)^{y-1}\Big|\sum_{l=0}^\infty c_l\frac{t^l}{l!}x^n}\\
&\quad -\frac{1}{n}\sum_{j=1}^r a_j\ang{\frac{\ln(1+t)}{(1+t)^{a_j}-1}\prod_{i=1}^r\left(\frac{\ln(1+t)}{(1+t)^{a_i}-1}\right)(1+t)^{y+a_j-1}\Big|\sum_{l=0}^\infty c_l\frac{t^l}{l!}x^n}\\
\end{align*}
\begin{align*}
&=\frac{r}{n}\sum_{l=0}^n c_l\binom{n}{l}\ang{\prod_{i=1}^r\left(\frac{\ln(1+t)}{(1+t)^{a_i}-1}\right)(1+t)^{y-1}\Big|x^{n-l}}\\
&\quad -\frac{1}{n}\sum_{j=1}^r a_j\sum_{l=0}^n c_l\binom{n}{l}\ang{\frac{\ln(1+t)}{(1+t)^{a_j}-1}\prod_{i=1}^r\left(\frac{\ln(1+t)}{(1+t)^{a_i}-1}\right)(1+t)^{y+a_j-1}\Big|x^{n-l}}\\
&=\frac{r}{n}\sum_{l=0}^n\binom{n}{l}c_l D_{n-l}(y-1|a_1,\dots,a_r)\\
&\quad -\frac{1}{n}\sum_{j=1}^r\sum_{l=0}^n \binom{n}{l}a_j c_l D_{n-l}(y+a_j-1|a_1,\dots,a_r,a_j)\,.
\end{align*}
Therefore, we obtain
\begin{align*}
D_n(x|a_1,\dots,a_r)
&=x D_{n-1}(x-1|a_1,\dots,a_r)\\
&\quad +\frac{r}{n}\sum_{l=0}^n\binom{n}{l}c_l D_{n-l}(x-1|a_1,\dots,a_r)\\
&\quad -\frac{1}{n}\sum_{j=1}^r\sum_{l=0}^n \binom{n}{l}a_j c_l D_{n-l}(x+a_j-1|a_1,\dots,a_r,a_j)\,,
\end{align*}
which is the identity (\ref{40}).

Next, for $n\ge 1$ we have
\begin{align*}
\widehat D_n(y|a_1,\dots,a_r)&=\ang{\sum_{l=0}^\infty \widehat D_l(y|a_1,\dots,a_r)\frac{t^l}{l!}\Big|x^n}\\
&=\ang{\prod_{j=1}^r\left(\frac{(1+t)^{a_j}\ln(1+t)}{(1+t)^{a_j}-1}\right)(1+t)^y\Big|x^n}\\
&=\ang{\partial_t\left(\prod_{j=1}^r\left(\frac{(1+t)^{a_j}\ln(1+t)}{(1+t)^{a_j}-1}\right)(1+t)^y\right)\Big|x^{n-1}}\\
&=\ang{\prod_{j=1}^r\left(\frac{(1+t)^{a_j}\ln(1+t)}{(1+t)^{a_j}-1}\right)\left(\partial_t(1+t)^y\right)\Big|x^{n-1}}\\
&\quad +\ang{\left(\partial_t\prod_{j=1}^r\left(\frac{(1+t)^{a_j}\ln(1+t)}{(1+t)^{a_j}-1}\right)\right)(1+t)^y\Big|x^{n-1}}\\
&=y \widehat D_{n-1}(y-1|a_1,\dots,a_r)\\
&\quad +\ang{\left(\partial_t\prod_{j=1}^r\left(\frac{(1+t)^{a_j}\ln(1+t)}{(1+t)^{a_j}-1}\right)\right)(1+t)^y\Big|x^{n-1}}\,.
\end{align*}
Observe that
\begin{align*}
&\partial_t\prod_{j=1}^r\left(\frac{(1+t)^{a_j}\ln(1+t)}{(1+t)^{a_j}-1}\right)\\
&=\partial_t\left(\prod_{j=1}^r\left(\frac{\ln(1+t)}{(1+t)^{a_j}-1}\right)\prod_{j=1}^r(1+t)^{a_j}\right)\\
&=\left(\partial_t\prod_{j=1}^r\left(\frac{\ln(1+t)}{(1+t)^{a_j}-1}\right)\right)\prod_{j=1}^r(1+t)^{a_j}\\
&\quad +\prod_{j=1}^r\left(\frac{\ln(1+t)}{(1+t)^{a_j}-1}\right)\left(\partial_t\prod_{j=1}^r(1+t)^{a_j}\right)\\
&=\frac{1}{1+t}\prod_{i=1}^r\left(\frac{(1+t)^{a_i}\ln(1+t)}{(1+t)^{a_i}-1}\right)\frac{\sum_{j=1}^r\left(\frac{t}{\ln(1+t)}-\frac{a_j t(1+t)^{a_j}}{(1+t)^{a_j}-1}\right)}{t}\\
&\quad +\frac{1}{1+t}\prod_{i=1}^r\left(\frac{(1+t)^{a_i}\ln(1+t)}{(1+t)^{a_i}-1}\right)\sum_{j=1}^r a_j\,.
\end{align*}
Thus, we have
\begin{align*}
&\ang{\left(\partial_t\prod_{i=1}^r\left(\frac{(1+t)^{a_i}\ln(1+t)}{(1+t)^{a_i}-1}\right)\right)(1+t)^y\Big|x^{n-1}}\\
&=\ang{\prod_{i=1}^r\left(\frac{(1+t)^{a_i}\ln(1+t)}{(1+t)^{a_i}-1}\right)(1+t)^{y-1}\Big|\frac{\sum_{j=1}^r\left(\frac{t}{\ln(1+t)}-\frac{a_j t(1+t)^{a_j}}{(1+t)^{a_j}-1}\right)}{t}x^{n-1}}\\
&\quad +\left(\sum_{j=1}^r a_j\right)\ang{\prod_{i=1}^r\left(\frac{(1+t)^{a_i}\ln(1+t)}{(1+t)^{a_i}-1}\right)(1+t)^{y-1}\Big|x^{n-1}}\\
&=\left(\sum_{j=1}^r a_j\right)\widehat D_{n-1}(y-1|a_1,\dots,a_r)\\
&\quad +\frac{1}{n}\ang{\prod_{i=1}^r\left(\frac{(1+t)^{a_i}\ln(1+t)}{(1+t)^{a_i}-1}\right)(1+t)^{y-1}\Big|\sum_{j=1}^r\left(\frac{t}{\ln(1+t)}-\frac{a_j t(1+t)^{a_j}}{(1+t)^{a_j}-1}\right)x^n}\\
\end{align*}
\begin{align*}
&=\left(\sum_{j=1}^r a_j\right)\widehat D_{n-1}(y-1|a_1,\dots,a_r)\\
&\quad +\frac{r}{n}\ang{\prod_{i=1}^r\left(\frac{(1+t)^{a_i}\ln(1+t)}{(1+t)^{a_i}-1}\right)(1+t)^{y-1}\Big|\frac{t}{\ln(1+t)}x^n}\\
&\quad -\frac{1}{n}\sum_{j=1}^r a_j\ang{\frac{(1+t)^{a_j}\ln(1+t)}{(1+t)^{a_j}-1}\prod_{i=1}^r\left(\frac{(1+t)^{a_i}\ln(1+t)}{(1+t)^{a_i}-1}\right)(1+t)^{y-1}\Big|\frac{t}{\ln(1+t)}x^n}\\
&=\left(\sum_{j=1}^r a_j\right)\widehat D_{n-1}(y-1|a_1,\dots,a_r)\\
&\quad +\frac{r}{n}\ang{\prod_{i=1}^r\left(\frac{(1+t)^{a_i}\ln(1+t)}{(1+t)^{a_i}-1}\right)(1+t)^{y-1}\Big|\sum_{l=0}^\infty c_l\frac{t^l}{l!}x^n}\\
&\quad -\frac{1}{n}\sum_{j=1}^r a_j\ang{\frac{(1+t)^{a_j}\ln(1+t)}{(1+t)^{a_j}-1}\prod_{i=1}^r\left(\frac{(1+t)^{a_i}\ln(1+t)}{(1+t)^{a_i}-1}\right)(1+t)^{y-1}\Big|\sum_{l=0}^\infty c_l\frac{t^l}{l!}x^n}\\
&=\left(\sum_{j=1}^r a_j\right)\widehat D_{n-1}(y-1|a_1,\dots,a_r)\\
&\quad +\frac{r}{n}\sum_{l=0}^n c_l\binom{n}{l}\ang{\prod_{i=1}^r\left(\frac{(1+t)^{a_i}\ln(1+t)}{(1+t)^{a_i}-1}\right)(1+t)^{y-1}\Big|x^{n-l}}\\
&\quad -\frac{1}{n}\sum_{j=1}^r a_j\sum_{l=0}^n c_l\binom{n}{l}\ang{\frac{(1+t)^{a_j}\ln(1+t)}{(1+t)^{a_j}-1}\prod_{i=1}^r\left(\frac{(1+t)^{a_i}\ln(1+t)}{(1+t)^{a_i}-1}\right)(1+t)^{y-1}\Big|x^{n-l}}\\
&=\left(\sum_{j=1}^r a_j\right)\widehat D_{n-1}(y-1|a_1,\dots,a_r)
+\frac{r}{n}\sum_{l=0}^n\binom{n}{l}c_l \widehat D_{n-l}(y-1|a_1,\dots,a_r)\\
&\quad -\frac{1}{n}\sum_{j=1}^r\sum_{l=0}^n \binom{n}{l}a_j c_l \widehat D_{n-l}(y-1|a_1,\dots,a_r,a_j)\,.
\end{align*}
Therefore, we obtain
\begin{align*}
\widehat D_n(x|a_1,\dots,a_r)
&=\left(x+\sum_{j=1}^r a_j\right)\widehat D_{n-1}(x-1|a_1,\dots,a_r)\\
&\quad +\frac{r}{n}\sum_{l=0}^n\binom{n}{l}c_l\widehat D_{n-l}(x-1|a_1,\dots,a_r)\\
&\quad -\frac{1}{n}\sum_{j=1}^r\sum_{l=0}^n \binom{n}{l}a_j c_l\widehat D_{n-l}(x-1|a_1,\dots,a_r,a_j)\,.
\end{align*}
which is the identity (\ref{42}).
\qed\end{proof}


\subsection{Relations including the Stirling numbers of the first kind}

\begin{theorem}
For $n-1\ge m\ge 1$, we have
\begin{align}
&\sum_{l=0}^{n-m}\binom{n}{l}S_1(n-l,m)D_l(a_1,\dots,a_r)\notag\\
&=\sum_{l=0}^{n-m}\binom{n-1}{l}S_1(n-l-1,m-1)D_l(-1|a_1,\dots,a_r)\notag\\
&\quad +\frac{1}{n}\sum_{l=0}^{n-m-1}\binom{n}{l+1}S_1(n-l-1,m)\notag\\
&\qquad \times\left(r\sum_{i=0}^{l+1}\binom{l+1}{i}c_i D_{l+1-i}(-1|a_1,\dots,a_r)\right.\notag\\
&\qquad\quad\left. -\sum_{j=1}^r\sum_{i=0}^{l+1}\binom{l+1}{i}a_j c_i D_{l+1-i}(a_j-1|a_1,\dots,a_r,a_j)\right)\,,
\label{50a}\\
&\sum_{l=0}^{n-m}\binom{n}{l}S_1(n-l,m)\widehat D_l(a_1,\dots,a_r)\notag\\
&=\sum_{l=0}^{n-m}\binom{n-1}{l}S_1(n-l-1,m-1)\widehat D_l(-1|a_1,\dots,a_r)\notag\\
&\quad +\frac{1}{n}\sum_{l=0}^{n-m-1}\binom{n}{l+1}S_1(n-l-1,m)\notag\\
&\qquad \times\left(r\sum_{i=0}^{l+1}\binom{l+1}{i}c_i\widehat D_{l+1-i}(-1|a_1,\dots,a_r)\right.\notag\\
&\qquad\quad\left. -\sum_{j=1}^r\sum_{i=0}^{l+1}\binom{l+1}{i}a_j c_i\widehat D_{l+1-i}(-1|a_1,\dots,a_r,a_j)\right)\notag\\
&\quad +\sum_{l=0}^{n-m-1}\binom{n-1}{l}S_1(n-l-1,m)\sum_{j=1}^r a_j\widehat D_l(-1|a_1,\dots,a_r)\,.
\label{50b}
\end{align}
\label{th50}
\end{theorem}
\begin{proof}
We shall compute
$$
\ang{\prod_{j=1}^r\left(\frac{\ln(1+t)}{(1+t)^{a_j}-1}\right)\bigl(\ln(1+t)\bigr)^m\Big|x^n}
$$
in two different ways.
On the one hand,
\begin{align*}
&\ang{\prod_{j=1}^r\left(\frac{\ln(1+t)}{(1+t)^{a_j}-1}\right)\bigl(\ln(1+t)\bigr)^m\Big|x^n}\\
&=\ang{\prod_{j=1}^r\left(\frac{\ln(1+t)}{(1+t)^{a_j}-1}\right)\Big|\bigl(\ln(1+t)\bigr)^m x^n}\\
&=\ang{\prod_{j=1}^r\left(\frac{\ln(1+t)}{(1+t)^{a_j}-1}\right)\Big|\sum_{l=0}^\infty\frac{m!}{(l+m)!}S_1(l+m,m)t^{l+m}x^n}\\
&=\sum_{l=0}^{n-m}\frac{m!}{(l+m)!}S_1(l+m,m)(n)_{l+m}\ang{\prod_{j=1}^r\left(\frac{\ln(1+t)}{(1+t)^{a_j}-1}\right)\Big|x^{n-l-m}}\\
&=\sum_{l=0}^{n-m}m!\binom{n}{l+m}S_1(l+m,m)D_{n-l-m}(a_1,\dots,a_r)\\
&=\sum_{l=0}^{n-m}m!\binom{n}{l}S_1(n-l,m)D_l(a_1,\dots,a_r)\,.
\end{align*}
On the other hand,
\begin{align}
&\ang{\prod_{j=1}^r\left(\frac{\ln(1+t)}{(1+t)^{a_j}-1}\right)\bigl(\ln(1+t)\bigr)^m\Big|x^n}\notag\\
&=\ang{\partial_t\left(\prod_{j=1}^r\left(\frac{\ln(1+t)}{(1+t)^{a_j}-1}\right)\bigl(\ln(1+t)\bigr)^m\right)\Big|x^{n-1}}\notag\\
&=\ang{\left(\partial_t\prod_{j=1}^r\left(\frac{\ln(1+t)}{(1+t)^{a_j}-1}\right)\right)\bigl(\ln(1+t)\bigr)^m\Big|x^{n-1}}\notag\\
&\quad +\ang{\prod_{j=1}^r\left(\frac{\ln(1+t)}{(1+t)^{a_j}-1}\right)\partial_t\left(\bigl(\ln(1+t)\bigr)^m\right)\Big|x^{n-1}}\,.
\label{53a}
\end{align}
The second term of (\ref{53a}) is equal to
\begin{align*}
&\ang{\prod_{j=1}^r\left(\frac{\ln(1+t)}{(1+t)^{a_j}-1}\right)\partial_t\left(\bigl(\ln(1+t)\bigr)^m\right)\Big|x^{n-1}}\\
&=m\ang{\prod_{j=1}^r\left(\frac{\ln(1+t)}{(1+t)^{a_j}-1}\right)(1+t)^{-1}\Big|\bigl(\ln(1+t)\bigr)^{m-1}x^{n-1}}\\
&=m\ang{\prod_{j=1}^r\left(\frac{\ln(1+t)}{(1+t)^{a_j}-1}\right)(1+t)^{-1}\Big|\sum_{l=0}^{n-m}\frac{(m-1)!}{(l+m-1)!}S_1(l+m-1,m-1)t^{l+m-1}x^{n-1}}\\
&=m\sum_{l=0}^{n-m}\frac{(m-1)!}{(l+m-1)!}S_1(l+m-1,m-1)(n-1)_{l+m-1}\\
&\quad \times\ang{\prod_{j=1}^r\left(\frac{\ln(1+t)}{(1+t)^{a_j}-1}\right)(1+t)^{-1}\Big|x^{n-l-m}}\\
&=m!\sum_{l=0}^{n-m}\binom{n-1}{l+m-1}S_1(l+m-1,m-1)D_{n-l-m}(-1|a_1,\dots,a_r)\\
&=m!\sum_{l=0}^{n-m}\binom{n-1}{l}S_1(n-l-1,m-1)D_l(-1|a_1,\dots,a_r)\,.
\end{align*}
The first term of (\ref{53a}) is equal to
\begin{align*}
&\ang{\left(\partial_t\prod_{j=1}^r\left(\frac{\ln(1+t)}{(1+t)^{a_j}-1}\right)\right)\bigl(\ln(1+t)\bigr)^m\Big|x^{n-1}}\\
&=\ang{\partial_t\prod_{j=1}^r\left(\frac{\ln(1+t)}{(1+t)^{a_j}-1}\right)\Big|\bigl(\ln(1+t)\bigr)^m x^{n-1}}\\
&=\ang{\partial_t\prod_{j=1}^r\left(\frac{\ln(1+t)}{(1+t)^{a_j}-1}\right)\Big|\sum_{l=0}^{n-m-1}\frac{m!}{(l+m)!}S_1(l+m,m)t^{l+m}x^{n-1}}\\
&=\sum_{l=0}^{n-m-1}\frac{m!}{(l+m)!}S_1(l+m,m)(n-1)_{l+m}\ang{\partial_t\prod_{j=1}^r\left(\frac{\ln(1+t)}{(1+t)^{a_j}-1}\right)\Big|x^{n-l-m-1}}\\
\end{align*}
\begin{align*}
&=\sum_{l=0}^{n-m-1}m!\binom{n-1}{l+m}S_1(l+m,m)\\
&\quad \times\ang{\prod_{i=1}^r\left(\frac{\ln(1+t)}{(1+t)^{a_i}-1}\right)(1+t)^{-1}\Big|\frac{\sum_{j=1}^r\left(\frac{t}{\ln(1+t)}-\frac{a_j t(1+t)^{a_j}}{(1+t)^{a_j}-1}\right)}{t}x^{n-l-m-1}}\\
&=m!\sum_{l=0}^{n-m-1}\frac{1}{n-l-m}\binom{n-1}{l+m}S_1(l+m,m)\\
&\quad \times\ang{\prod_{i=1}^r\left(\frac{\ln(1+t)}{(1+t)^{a_i}-1}\right)(1+t)^{-1}\Big|\sum_{j=1}^r\left(\frac{t}{\ln(1+t)}-\frac{a_j t(1+t)^{a_j}}{(1+t)^{a_j}-1}\right)x^{n-l-m}}\\
&=\frac{m!}{n}\sum_{l=0}^{n-m-1}\binom{n}{l+1}S_1(n-1-l,m)\\
&\quad \times\left(r\ang{\prod_{i=1}^r\left(\frac{\ln(1+t)}{(1+t)^{a_i}-1}\right)(1+t)^{-1}\Big|\frac{t}{\ln(1+t)}x^{l+1}}\right.\\
&\qquad\left. -\left(\sum_{j=1}^r a_j\right)\ang{\frac{\ln(1+t)}{(1+t)^{a_j}-1}(1+t)^{a_j-1}\prod_{i=1}^r\left(\frac{\ln(1+t)}{(1+t)^{a_i}-1}\right)\Big|\frac{t}{\ln(1+t)}x^{l+1}}\right)\\
&=\frac{m!}{n}\sum_{l=0}^{n-m-1}\binom{n}{l+1}S_1(n-l-1,m)\\
&\quad \times\left(r\ang{\prod_{i=1}^r\left(\frac{\ln(1+t)}{(1+t)^{a_i}-1}\right)(1+t)^{-1}\Big|\sum_{i=0}^{l+1}c_i\frac{t^i}{i!}x^{l+1}}\right.\\
&\qquad\left. -\left(\sum_{j=1}^r a_j\right)\ang{\frac{\ln(1+t)}{(1+t)^{a_j}-1}(1+t)^{a_j-1}\prod_{i=1}^r\left(\frac{\ln(1+t)}{(1+t)^{a_i}-1}\right)\Big|\sum_{i=0}^{l+1}c_i\frac{t^i}{i!}x^{l+1}}\right)\\
&=\frac{m!}{n}\sum_{l=0}^{n-m-1}\binom{n}{l+1}S_1(n-l-1,m)\\
&\quad \times\left(r\sum_{i=0}^{l+1}\binom{l+1}{i}c_i D_{l+1-i}(-1|a_1,\dots,a_r)\right.\\
&\qquad\left. -\sum_{j=1}^r a_j\sum_{i=0}^{l+1}\binom{l+1}{i}c_i D_{l+1-i}(a_j-1|a_1,\dots,a_r,a_j)\right)\,.
\end{align*}
Therefore, we have, for $n-1\ge m\ge 1$,
\begin{align*}
&m!\sum_{l=0}^{n-m}\binom{n}{l}S_1(n-l,m)D_l(a_1,\dots,a_r)\\
&=m!\sum_{l=0}^{n-m}\binom{n-1}{l}S_1(n-l-1,m-1)D_l(-1|a_1,\dots,a_r)\\
&\quad +\frac{m!}{n}\sum_{l=0}^{n-m-1}\binom{n}{l+1}S_1(n-l-1,m)\\
&\qquad \times\left(r\sum_{i=0}^{l+1}\binom{l+1}{i}c_{l+1-i}D_i(-1|a_1,\dots,a_r)\right.\\
&\qquad\quad\left. -\sum_{j=1}^r\sum_{i=0}^{l+1} a_j\binom{l+1}{i}c_{l+1-i}D_i(a_j-1|a_1,\dots,a_r,a_j)\right)\,.
\end{align*}
Thus, we get (\ref{50a}).

Next, we shall compute
$$
\ang{\prod_{j=1}^r\left(\frac{(1+t)^{a_j}\ln(1+t)}{(1+t)^{a_j}-1}\right)\bigl(\ln(1+t)\bigr)^m\Big|x^n}
$$
in two different ways.
On the one hand,
\begin{align*}
&\ang{\prod_{j=1}^r\left(\frac{(1+t)^{a_j}\ln(1+t)}{(1+t)^{a_j}-1}\right)\bigl(\ln(1+t)\bigr)^m\Big|x^n}\\
&=\ang{\prod_{j=1}^r\left(\frac{(1+t)^{a_j}\ln(1+t)}{(1+t)^{a_j}-1}\right)\Big|\bigl(\ln(1+t)\bigr)^m x^n}\\
&=\ang{\prod_{j=1}^r\left(\frac{(1+t)^{a_j}\ln(1+t)}{(1+t)^{a_j}-1}\right)\Big|\sum_{l=0}^\infty\frac{m!}{(l+m)!}S_1(l+m,m)t^{l+m}x^n}\\
&=\sum_{l=0}^{n-m}\frac{m!}{(l+m)!}S_1(l+m,m)(n)_{l+m}\ang{\prod_{j=1}^r\left(\frac{(1+t)^{a_j}\ln(1+t)}{(1+t)^{a_j}-1}\right)\Big|x^{n-l-m}}\\
&=\sum_{l=0}^{n-m}m!\binom{n}{l+m}S_1(l+m,m)\widehat D_{n-l-m}(a_1,\dots,a_r)\\
&=\sum_{l=0}^{n-m}m!\binom{n}{l}S_1(n-l,m)\widehat D_l(a_1,\dots,a_r)\,.
\end{align*}
On the other hand,
\begin{align}
&\ang{\prod_{j=1}^r\left(\frac{(1+t)^{a_j}\ln(1+t)}{(1+t)^{a_j}-1}\right)\bigl(\ln(1+t)\bigr)^m\Big|x^n}\notag\\
&=\ang{\partial_t\left(\prod_{j=1}^r\left(\frac{(1+t)^{a_j}\ln(1+t)}{(1+t)^{a_j}-1}\right)\bigl(\ln(1+t)\bigr)^m\right)\Big|x^{n-1}}\notag\\
&=\ang{\left(\partial_t\prod_{j=1}^r\left(\frac{(1+t)^{a_j}\ln(1+t)}{(1+t)^{a_j}-1}\right)\right)\bigl(\ln(1+t)\bigr)^m\Big|x^{n-1}}\notag\\
&\quad +\ang{\prod_{j=1}^r\left(\frac{(1+t)^{a_j}\ln(1+t)}{(1+t)^{a_j}-1}\right)\partial_t\left(\bigl(\ln(1+t)\bigr)^m\right)\Big|x^{n-1}}\,.
\label{53b}
\end{align}
The second term of (\ref{53b}) is equal to
\begin{align*}
&\ang{\prod_{j=1}^r\left(\frac{(1+t)^{a_j}\ln(1+t)}{(1+t)^{a_j}-1}\right)\partial_t\left(\bigl(\ln(1+t)\bigr)^m\right)\Big|x^{n-1}}\\
&=m\ang{\prod_{j=1}^r\left(\frac{(1+t)^{a_j}\ln(1+t)}{(1+t)^{a_j}-1}\right)(1+t)^{-1}\Big|\bigl(\ln(1+t)\bigr)^{m-1}x^{n-1}}\\
&=m\ang{\prod_{j=1}^r\left(\frac{(1+t)^{a_j}\ln(1+t)}{(1+t)^{a_j}-1}\right)(1+t)^{-1}\Big|\sum_{l=0}^{n-m}\frac{(m-1)!}{(l+m-1)!}S_1(l+m-1,m-1)t^{l+m-1}x^{n-1}}\\
&=m\sum_{l=0}^{n-m}\frac{(m-1)!}{(l+m-1)!}S_1(l+m-1,m-1)(n-1)_{l+m-1}\\
&\quad \times\ang{\prod_{j=1}^r\left(\frac{(1+t)^{a_j}\ln(1+t)}{(1+t)^{a_j}-1}\right)(1+t)^{-1}\Big|x^{n-l-m}}\\
&=m!\sum_{l=0}^{n-m}\binom{n-1}{l+m-1}S_1(l+m-1,m-1)\widehat D_{n-l-m}(-1|a_1,\dots,a_r)\\
&=m!\sum_{l=0}^{n-m}\binom{n-1}{l}S_1(n-l-1,m-1)\widehat D_l(-1|a_1,\dots,a_r)\,.
\end{align*}
The first term of (\ref{53b}) is equal to
\begin{align*}
&\ang{\left(\partial_t\prod_{j=1}^r\left(\frac{(1+t)^{a_j}\ln(1+t)}{(1+t)^{a_j}-1}\right)\right)\bigl(\ln(1+t)\bigr)^m\Big|x^{n-1}}\\
&=\ang{\partial_t\prod_{j=1}^r\left(\frac{(1+t)^{a_j}\ln(1+t)}{(1+t)^{a_j}-1}\right)\Big|\bigl(\ln(1+t)\bigr)^m x^{n-1}}\\
&=\ang{\partial_t\prod_{j=1}^r\left(\frac{(1+t)^{a_j}\ln(1+t)}{(1+t)^{a_j}-1}\right)\Big|\sum_{l=0}^{n-m-1}\frac{m!}{(l+m)!}S_1(l+m,m)t^{l+m}x^{n-1}}\\
&=\sum_{l=0}^{n-m-1}\frac{m!}{(l+m)!}S_1(l+m,m)(n-1)_{l+m}\ang{\partial_t\prod_{j=1}^r\left(\frac{(1+t)^{a_j}\ln(1+t)}{(1+t)^{a_j}-1}\right)\Big|x^{n-l-m-1}}\,.
\end{align*}
From the proof of (\ref{42}), we recall
\begin{align*}
\partial_t\prod_{j=1}^r\left(\frac{(1+t)^{a_j}\ln(1+t)}{(1+t)^{a_j}-1}\right)
&=\frac{1}{1+t}\prod_{i=1}^r\left(\frac{(1+t)^{a_i}\ln(1+t)}{(1+t)^{a_i}-1}\right)\frac{\sum_{j=1}^r\left(\frac{t}{\ln(1+t)}-\frac{a_j t(1+t)^{a_j}}{(1+t)^{a_j}-1}\right)}{t}\\
&\quad +\frac{1}{1+t}\prod_{i=1}^r\left(\frac{(1+t)^{a_i}\ln(1+t)}{(1+t)^{a_i}-1}\right)\sum_{j=1}^r a_j\,.
\end{align*}
Hence, the first term of (\ref{53a}) is equal to
\begin{align*}
&\sum_{l=0}^{n-m-1}m!\binom{n-1}{l+m}S_1(l+m,m)\\
&\quad \times\left(\ang{\prod_{i=1}^r\left(\frac{(1+t)^{a_i}\ln(1+t)}{(1+t)^{a_i}-1}\right)(1+t)^{-1}\Big|\frac{\sum_{j=1}^r\left(\frac{t}{\ln(1+t)}-\frac{a_j t(1+t)^{a_j}}{(1+t)^{a_j}-1}\right)}{t}x^{n-l-m-1}}\right.\\
&\qquad\left. +\left(\sum_{j=1}^r a_j\right)\ang{\prod_{i=1}^r\left(\frac{(1+t)^{a_i}\ln(1+t)}{(1+t)^{a_i}-1}\right)(1+t)^{-1}\Big|x^{n-l-m-1}}\right)\\
&=m!\sum_{l=0}^{n-m-1}\binom{n-1}{l}S_1(n-l-1,m)\\
&\quad \times\left(\ang{\prod_{i=1}^r\left(\frac{(1+t)^{a_i}\ln(1+t)}{(1+t)^{a_i}-1}\right)(1+t)^{-1}\Big|\frac{\sum_{j=1}^r\left(\frac{t}{\ln(1+t)}-\frac{a_j t(1+t)^{a_j}}{(1+t)^{a_j}-1}\right)}{t}x^l}\right.\\
&\qquad\left. +\left(\sum_{j=1}^r a_j\right)\ang{\prod_{i=1}^r\left(\frac{(1+t)^{a_i}\ln(1+t)}{(1+t)^{a_i}-1}\right)(1+t)^{-1}\Big|x^l}\right)\\
&=m!\sum_{l=0}^{n-m-1}\binom{n-1}{l}S_1(n-l-1,m)\\
&\quad \times\left(\frac{1}{l+1}\ang{\prod_{i=1}^r\left(\frac{(1+t)^{a_i}\ln(1+t)}{(1+t)^{a_i}-1}\right)(1+t)^{-1}\Big|\sum_{j=1}^r\left(\frac{t}{\ln(1+t)}-\frac{a_j t(1+t)^{a_j}}{(1+t)^{a_j}-1}\right)x^{l+1}}\right.\\
&\qquad\left. +\left(\sum_{j=1}^r a_j\right)\ang{\prod_{i=1}^r\left(\frac{(1+t)^{a_i}\ln(1+t)}{(1+t)^{a_i}-1}\right)(1+t)^{-1}\Big|x^l}\right)\\
&=m!\sum_{l=0}^{n-m-1}\binom{n-1}{l}S_1(n-l-1,m)\\
&\quad \times\left(\frac{r}{l+1}\ang{\prod_{i=1}^r\left(\frac{(1+t)^{a_i}\ln(1+t)}{(1+t)^{a_i}-1}\right)(1+t)^{-1}\Big|\frac{t}{\ln(1+t)}x^{l+1}}\right.\\
&\qquad -\frac{1}{l+1}\left(\sum_{j=1}^r a_j\right)\ang{\frac{(1+t)^{a_j}\ln(1+t)}{(1+t)^{a_j}-1}\prod_{i=1}^r\left(\frac{(1+t)^{a_i}\ln(1+t)}{(1+t)^{a_i}-1}\right)(1+t)^{-1}\Big|\frac{t}{\ln(1+t)}x^{l+1}}\\
&\qquad\left. +\left(\sum_{j=1}^r a_j\right)\ang{\prod_{i=1}^r\left(\frac{(1+t)^{a_i}\ln(1+t)}{(1+t)^{a_i}-1}\right)(1+t)^{-1}\Big|x^l}\right)\\
\end{align*}
\begin{align*}
&=m!\sum_{l=0}^{n-m-1}\binom{n-1}{l}S_1(n-l-1,m)\\
&\quad \times\left(\frac{r}{l+1}\ang{\prod_{i=1}^r\left(\frac{(1+t)^{a_i}\ln(1+t)}{(1+t)^{a_i}-1}\right)(1+t)^{-1}\Big|\sum_{i=0}^{l+1}c_i\frac{t^i}{i!}x^{l+1}}\right.\\
&\qquad -\frac{1}{l+1}\left(\sum_{j=1}^r a_j\right)\ang{\frac{(1+t)^{a_j}\ln(1+t)}{(1+t)^{a_j}-1}\prod_{i=1}^r\left(\frac{(1+t)^{a_i}\ln(1+t)}{(1+t)^{a_i}-1}\right)(1+t)^{-1}\Big|\sum_{i=0}^{l+1}c_i\frac{t^i}{i!}x^{l+1}}\\
&\qquad\left. +\left(\sum_{j=1}^r a_j\right)\ang{\prod_{i=1}^r\left(\frac{(1+t)^{a_i}\ln(1+t)}{(1+t)^{a_i}-1}\right)(1+t)^{-1}\Big|x^l}\right)\\
&=m!\sum_{l=0}^{n-m-1}\binom{n-1}{l}S_1(n-l-1,m)\\
&\quad \times\left(\frac{r}{l+1}\sum_{i=0}^{l+1}\binom{l+1}{i}c_i\widehat D_{l+1-i}(-1|a_1,\dots,a_r)\right.\\
&\qquad\left. -\frac{1}{l+1}\sum_{j=1}^r a_j\sum_{i=0}^{l+1}\binom{l+1}{i}c_i\widehat D_{l+1-i}(-1|a_1,\dots,a_r,a_j)
+\sum_{j=1}^r a_j\widehat D_l(-1|a_1,\dots,a_r)\right)\\
&=\frac{m!}{n}\sum_{l=0}^{n-m-1}\binom{n}{l+1}S_1(n-l-1,m)\\
&\quad \times\left(r\sum_{i=0}^{l+1}\binom{l+1}{i}c_i\widehat D_{l+1-i}(-1|a_1,\dots,a_r)\right.\\
&\qquad\left. -\sum_{j=1}^r \sum_{i=0}^{l+1}\binom{l+1}{i}a_j c_i\widehat D_{l+1-i}(-1|a_1,\dots,a_r,a_j)\right)\\
&\quad +m!\sum_{l=0}^{n-m-1}\binom{n-1}{l}S_1(n-l-1,m)\sum_{j=1}^r a_j\widehat D_l(-1|a_1,\dots,a_r)\,.
\end{align*}
Therefore, we get (\ref{50b}).
\qed\end{proof}


\subsection{Relations with the falling factorials}

\begin{theorem}
\begin{align}
D_n(x|a_1,\dots,a_r)=\sum_{m=0}^n\binom{n}{m}D_{n-m}(a_1,\dots,a_r)(x)_m\,,
\label{60}\\
\widehat D_n(x|a_1,\dots,a_r)=\sum_{m=0}^n\binom{n}{m}\widehat D_{n-m}(a_1,\dots,a_r)(x)_m\,.
\label{61}\
\end{align}
\label{th60}
\end{theorem}
\begin{proof}
For (\ref{10b}) and (\ref{51}), assume that
$D_n(x|a_1,\dots,a_r)=\sum_{m=0}^n C_{n,m}(x)_m$. By (\ref{uc19}), we have
\begin{align*}
C_{n,m}&=\frac{1}{m!}\ang{\frac{1}{\prod_{j=1}^r\left(\frac{e^{a_j\ln(1+t)}-1}{\ln(1+t)}\right)}t^m\Big|x^n}\\
&=\frac{1}{m!}\ang{\prod_{j=1}^r\left(\frac{\ln(1+t)}{(1+t)^{a_j}-1}\right)\Big|t^m x^n}\\
&=\binom{n}{m}\ang{\prod_{j=1}^r\left(\frac{\ln(1+t)}{(1+t)^{a_j}-1}\right)\Big|x^{n-m}}\\
&=\binom{n}{m}D_{n-m}(a_1,\dots,a_r)\,.
\end{align*}
Thus, we get the identity (\ref{60}).

Similarly, for (\ref{10b}) and (\ref{51}), assume that
$\widehat D_n(x|a_1,\dots,a_r)=\sum_{m=0}^n C_{n,m}(x)_m$. By (\ref{uc19}), we have
\begin{align*}
C_{n,m}&=\frac{1}{m!}\ang{\frac{1}{\prod_{j=1}^r\left(\frac{e^{a_j\ln(1+t)}-1}{e^{a_j\ln(1+t)}\ln(1+t)}\right)}t^m\Big|x^n}\\
&=\frac{1}{m!}\ang{\prod_{j=1}^r\left(\frac{(1+t)^{a_j}\ln(1+t)}{(1+t)^{a_j}-1}\right)\Big|t^m x^n}\\
&=\binom{n}{m}\ang{\prod_{j=1}^r\left(\frac{(1+t)^{a_j}\ln(1+t)}{(1+t)^{a_j}-1}\right)\Big|x^{n-m}}\\
&=\binom{n}{m}\widehat D_{n-m}(a_1,\dots,a_r)\,.
\end{align*}
Thus, we get the identity (\ref{61}).
\qed\end{proof}


\subsection{Relations with higher-order Frobenius-Euler polynomials}

For $\lambda\in\mathbb C$ with $\lambda\ne 1$, the Frobenius-Euler polynomials of order $r$, $H_{n}^{(r)}(x|\lambda)$ are defined by the generating function
$$
\left(\frac{1-\lambda}{e^t-\lambda}\right)^r e^{xt}=\sum_{n=0}^\infty H_{n}^{(r)}(x|\lambda)\frac{t^n}{n!}
$$
(see e.g. \cite{KimKim1, KKL}).

\begin{theorem}
\begin{align}
&D_n(x|a_1,\dots,a_r)
=\sum_{m=0}^n\left(\sum_{j=0}^{n-m}\sum_{l=0}^{n-m-j}\binom{s}{j}\binom{n-j}{l}(n)_j\right.\notag\\
&\qquad\qquad\qquad\left. \times(1-\lambda)^{-j}S_1(n-j-l,m)D_l(a_1,\dots,a_r)\right)H_m^{(s)}(x|\lambda)\,,
\label{80}\\
&\widehat D_n(x|a_1,\dots,a_r)
=\sum_{m=0}^n\left(\sum_{j=0}^{n-m}\sum_{l=0}^{n-m-j}\binom{s}{j}\binom{n-j}{l}(n)_j\right.\notag\\
&\qquad\qquad\qquad\left. \times(1-\lambda)^{-j}S_1(n-j-l,m)\widehat D_l(a_1,\dots,a_r)\right)H_m^{(s)}(x|\lambda)\,.
\label{81}
\end{align}
\label{th80}
\end{theorem}
\begin{proof}
For (\ref{10b}) and
\begin{equation}
H_n^{(s)}(x|\lambda)\sim\left(\left(\frac{e^t-\lambda}{1-\lambda}\right)^s, t\right)\,,
\label{86}
\end{equation}
assume that
$D_n(x|a_1,\dots,a_r)=\sum_{m=0}^n C_{n,m}H_m^{(s)}(x|\lambda)$. By (\ref{uc19}), similarly to the proof of (\ref{50a}), we have
\begin{align*}
C_{n,m}&=\frac{1}{m!}\ang{\frac{\left(\frac{e^{\ln(1+t)}-\lambda}{1-\lambda}\right)^s}{\prod_{j=1}^r\left(\frac{e^{a_j\ln(1+t)}-1}{\ln(1+t)}\right)}\bigl(\ln(1+t)\bigr)^m\Big|x^n}\\
&=\frac{1}{m!(1-\lambda)^s}\ang{\prod_{j=1}^r\left(\frac{\ln(1+t)}{(1+t)^{a_j}-1}\right)\bigl(\ln(1+t)\bigr)^m(1-\lambda+t)^s\Big|x^n}\\
&=\frac{1}{m!(1-\lambda)^s}\ang{\prod_{j=1}^r\left(\frac{\ln(1+t)}{(1+t)^{a_j}-1}\right)\bigl(\ln(1+t)\bigr)^m\Big|\sum_{i=0}^{\min\{s,n\}}\binom{s}{i}(1-\lambda)^{s-i}t^i x^n}\\
&=\frac{1}{m!(1-\lambda)^s}\sum_{i=0}^{n-m}\binom{s}{i}(1-\lambda)^{s-i}(n)_i\ang{\prod_{j=1}^r\left(\frac{\ln(1+t)}{(1+t)^{a_j}-1}\right)\Big|\bigl(\ln(1+t)\bigr)^m x^{n-i}}\\
&=\frac{1}{m!(1-\lambda)^s}\sum_{i=0}^{n-m}\binom{s}{i}(1-\lambda)^{s-i}(n)_i\sum_{l=0}^{n-m-i}m!\binom{n-i}{l}S_1(n-i-l,m)D_l(a_1,\dots,a_r)\\
&=\sum_{i=0}^{n-m}\sum_{l=0}^{n-m-i}\binom{s}{i}\binom{n-i}{l}(n)_i(1-\lambda)^{-i}S_1(n-i-l,m)D_l(a_1,\dots,a_r)\,.
\end{align*}
Thus, we get the identity (\ref{80}).

Next, for (\ref{11b}) and (\ref{86}),
assume that
$\widehat D_n(x|a_1,\dots,a_r)=\sum_{m=0}^n C_{n,m}H_m^{(s)}(x|\lambda)$. By (\ref{uc19}), similarly to the proof of (\ref{50b}), we have
\begin{align*}
C_{n,m}&=\frac{1}{m!}\ang{\frac{\left(\frac{e^{\ln(1+t)}-\lambda}{1-\lambda}\right)^s}{\prod_{j=1}^r\left(\frac{e^{a_j\ln(1+t)}-1}{e^{a_j\ln(1+t)}\ln(1+t)}\right)}\bigl(\ln(1+t)\bigr)^m\Big|x^n}\\
&=\frac{1}{m!(1-\lambda)^s}\ang{\prod_{j=1}^r\left(\frac{(1+t)^{a_j}\ln(1+t)}{(1+t)^{a_j}-1}\right)\bigl(\ln(1+t)\bigr)^m\Big|(1-\lambda+t)^s x^n}\\
&=\frac{1}{m!(1-\lambda)^s}\ang{\prod_{j=1}^r\left(\frac{(1+t)^{a_j}\ln(1+t)}{(1+t)^{a_j}-1}\right)\bigl(\ln(1+t)\bigr)^m\Big|\sum_{i=0}^{\min\{s,n\}}\binom{s}{i}(1-\lambda)^{s-i}t^i x^n}\\
&=\frac{1}{m!(1-\lambda)^s}\sum_{i=0}^{n-m}\binom{s}{i}(1-\lambda)^{s-i}(n)_i\ang{\prod_{j=1}^r\left(\frac{(1+t)^{a_j}\ln(1+t)}{(1+t)^{a_j}-1}\right)\Big|\bigl(\ln(1+t)\bigr)^m x^{n-i}}\\
&=\frac{1}{m!(1-\lambda)^s}\sum_{i=0}^{n-m}\binom{s}{i}(1-\lambda)^{s-i}(n)_i\sum_{l=0}^{n-m-i}m!\binom{n-i}{l}S_1(n-i-l,m)\widehat D_l(a_1,\dots,a_r)\\
&=\sum_{i=0}^{n-m}\sum_{l=0}^{n-m-i}\binom{s}{i}\binom{n-i}{l}(n)_i(1-\lambda)^{-i}S_1(n-i-l,m)\widehat D_l(a_1,\dots,a_r)\,.
\end{align*}
Thus, we get the identity (\ref{81}).
\qed\end{proof}


\subsection{Relations with higher-order Bernoulli polynomials}

Bernoulli polynomials $\mathfrak B_n^{(r)}(x)$ of order $r$ are defined by
$$
\left(\frac{t}{e^t-1}\right)^r e^{xt}=\sum_{n=0}^\infty\frac{\mathfrak B_n^{(r)}(x)}{n!}t^n
$$
(see e.g. \cite[Section 2.2]{Roman}).
In addition, Cauchy numbers of the first kind $\mathfrak C_n^{(r)}$ of order $r$ are defined by
$$
\left(\frac{t}{\ln(1+t)}\right)^r=\sum_{n=0}^\infty\frac{\mathfrak C_n^{(r)}}{n!}t^n
$$
(see e.g. \cite[(2.1)]{Car}, \cite[(6)]{LW}).

\begin{theorem}
\begin{align}
&D_n(x|a_1,\dots,a_r)\notag\\
&=\sum_{m=0}^n\left(\sum_{i=0}^{n-m}\sum_{l=0}^{n-m-i}\binom{n}{i}\binom{n-i}{l}\mathfrak C_i^{(s)}S_1(n-i-l,m)D_l(a_1,\dots,a_r)\right)\mathfrak B_m^{(s)}(x)\,,
\label{90}\\
&\widehat D_n(x|a_1,\dots,a_r)\notag\\
&=\sum_{m=0}^n\left(\sum_{i=0}^{n-m}\sum_{l=0}^{n-m-i}\binom{n}{i}\binom{n-i}{l}\mathfrak C_i^{(s)}S_1(n-i-l,m)\widehat D_l(a_1,\dots,a_r)\right)\mathfrak B_m^{(s)}(x)\,.
\label{91}
\end{align}
\label{th90}
\end{theorem}
\begin{proof}
For (\ref{10b}) and
\begin{equation}
\mathfrak B_n^{(s)}(x)\sim\left(\left(\frac{e^t-1}{t}\right)^s, t\right)\,,
\label{96}
\end{equation}
assume that $D_n(x|a_1,\dots,a_r)=\sum_{m=0}^n C_{n,m}\mathfrak B_m^{(s)}(x)$.
By (\ref{uc19}), similarly to the proof of (\ref{50a}), we have
\begin{align*}
C_{n,m}&=\frac{1}{m!}\ang{\frac{\left(\frac{e^{\ln(1+t)}-1}{\ln(1+t)}\right)^s}{\prod_{j=1}^r\left(\frac{e^{a_j\ln(1+t)}-1}{\ln(1+t)}\right)}\bigl(\ln(1+t)\bigr)^m\Big|x^n}\\
&=\frac{1}{m!}\ang{\prod_{j=1}^r\left(\frac{\ln(1+t)}{(1+t)^{a_j}-1}\right)  \bigl(\ln(1+t)\bigr)^m\Big|\left(\frac{t}{\ln(1+t)}\right)^s x^n}\\
&=\frac{1}{m!}\ang{\prod_{j=1}^r\left(\frac{\ln(1+t)}{(1+t)^{a_j}-1}\right)  \bigl(\ln(1+t)\bigr)^m\Big|\sum_{i=0}^\infty\mathfrak C_i^{(s)}\frac{t^i}{i!}x^n}\\
&=\frac{1}{m!}\sum_{i=0}^{n-m}\mathfrak C_i^{(s)}\binom{n}{i}\ang{\prod_{j=1}^r\left(\frac{\ln(1+t)}{(1+t)^{a_j}-1}\right)  \bigl(\ln(1+t)\bigr)^m\Big|x^{n-i}}\\
&=\frac{1}{m!}\sum_{i=0}^{n-m}\mathfrak C_i^{(s)}\binom{n}{i}\sum_{l=0}^{n-m-i}m!\binom{n-i}{l}S_1(n-i-l,m)D_l(a_1,\dots,a_r)\\
&=\sum_{i=0}^{n-m}\sum_{l=0}^{n-m-i}\binom{n}{i}\binom{n-i}{l}\mathfrak C_i^{(s)}S_1(n-i-l,m)D_l(a_1,\dots,a_r)\,.
\end{align*}
Thus, we get the identity (\ref{90}).

Next, for (\ref{10b}) and (\ref{96}), assume that $\widehat D_n(x|a_1,\dots,a_r)=\sum_{m=0}^n C_{n,m}\mathfrak B_m^{(s)}(x)$.
By (\ref{uc19}), similarly to the proof of (\ref{50b}), we have
\begin{align*}
C_{n,m}&=\frac{1}{m!}\ang{\frac{\left(\frac{e^{\ln(1+t)}-1}{\ln(1+t)}\right)^s}{\prod_{j=1}^r\left(\frac{e^{a_j\ln(1+t)}-1}{e^{a_j\ln(1+t)}\ln(1+t)}\right)}\bigl(\ln(1+t)\bigr)^m\Big|x^n}\\
&=\frac{1}{m!}\ang{\prod_{j=1}^r\left(\frac{(1+t)^{a_j}\ln(1+t)}{(1+t)^{a_j}-1}\right)  \bigl(\ln(1+t)\bigr)^m\Big|\left(\frac{t}{\ln(1+t)}\right)^s x^n}\\&=\frac{1}{m!}\ang{\prod_{j=1}^r\left(\frac{(1+t)^{a_j}\ln(1+t)}{(1+t)^{a_j}-1}\right)  \bigl(\ln(1+t)\bigr)^m\Big|\sum_{i=0}^\infty\mathfrak C_i^{(s)}\frac{t^i}{i!}x^n}\\
&=\frac{1}{m!}\sum_{i=0}^{n-m}\mathfrak C_i^{(s)}\binom{n}{i}\ang{\prod_{j=1}^r\left(\frac{(1+t)^{a_j}\ln(1+t)}{(1+t)^{a_j}-1}\right)  \bigl(\ln(1+t)\bigr)^m\Big|x^{n-i}}\\
&=\frac{1}{m!}\sum_{i=0}^{n-m}\mathfrak C_i^{(s)}\binom{n}{i}\sum_{l=0}^{n-m-i}m!\binom{n-i}{l}S_1(n-i-l,m)\widehat D_l(a_1,\dots,a_r)\\
&=\sum_{i=0}^{n-m}\sum_{l=0}^{n-m-i}\binom{n}{i}\binom{n-i}{l}\mathfrak C_i^{(s)}S_1(n-i-l,m)\widehat D_l(a_1,\dots,a_r)\,.
\end{align*}
Thus, we get the identity (\ref{91}).
\qed\end{proof}

\end{document}